\documentclass[A4]{article}
\usepackage{amscd}
\usepackage{amssymb}
\usepackage{amsmath}
\usepackage{amsthm}
\usepackage{latexsym}
\usepackage{upref}
\usepackage{epsfig}
\usepackage{mathrsfs}
\usepackage{float}
\usepackage{indentfirst}
\usepackage{hyperref}
\hypersetup{
    colorlinks=true,
    linkcolor=blue,
    citecolor=blue,
    filecolor=blue,
    urlcolor=blue,
}
\usepackage{graphics,graphicx,color}
\usepackage{floatflt}
\usepackage{cite}
\usepackage{enumitem}
\usepackage{fancyhdr} 
\usepackage{mathptmx} 
\usepackage{rotating} 
\usepackage{booktabs} 
\usepackage{pb-diagram,pb-xy}       
\usepackage{tikz}                   
\usetikzlibrary{matrix,arrows}
\usepackage{marginnote}

\renewcommand\footnotemark{}

\theoremstyle{plain}
\newtheorem{theorem}{Theorem}[section]
\newtheorem{lemma}[theorem]{Lemma}
\newtheorem{proposition}[theorem]{Proposition}

\newtheorem{corollary}[theorem]{Corollary}
\theoremstyle{definition}
\newtheorem{definition}[theorem]{Definition}

\theoremstyle{remark}

\makeatletter
\def\th@plain{%
  \thm@notefont{}
  \itshape 
}
\def\th@definition{%
  \thm@notefont{}
  \normalfont 
} \makeatother

\setlist{font=\normalfont}

\DeclareMathAlphabet{\mathcal}{OMS}{cmsy}{m}{n} %

\newcommand{\set}[1]{\{#1\}}
\newcommand{\cset}[2]{\set{{#1}\colon{#2}}}
\newcommand{\lprod}[2]{\displaystyle\prod_{#1}^{#2}}
\newcommand{\lcup}[2]{\displaystyle\bigcup_{#1}^{#2}}
\newcommand{\gyr}[2]{\mathrm{gyr}\,{[{#1}]{#2}}}
\newcommand{\lcap}[2]{\displaystyle\bigcap_{#1}^{#2}}
\newcommand{\gen}[1]{\langle#1\rangle}
\newcommand{\NC}[1]{\gen{\overline{#1}}}
\newcommand{\cols}[1]{\mathbf{\mathcal{#1}}}
\newcommand{\aut}[1]{\mathrm{Aut}\,{(#1)}}
\newcommand{\id}[1]{\mathrm{id}_{#1}}
\newcommand{\igyr}[2]{{\mathrm{gyr^{-1}}[{#1}]}{#2}}
\newcommand{\res}[2]{{#1}\hskip-3pt\mid_{#2}}
\newcommand{\lmlt}[1]{\mathrm{LMlt}\,{(#1)}}
\newcommand{\sym}[1]{\mathrm{Sym}\,{(#1)}}
\newcommand{\symz}[1]{\mathrm{Sym}_0({#1})}
\newcommand{\rad}[1]{\mathrm{Rad}\,{(#1)}}
\newcommand{\abs}[1]{|#1|}
\newcommand{\Z}{\mathbb{Z}}

\DeclareMathOperator{\im}{Im}

\newcommand{\vphi}{\varphi}
\newcommand{\Gam}{\Gamma}
\newcommand{\sig}{\sigma}

\renewcommand{\setminus}{-}

\begin{document}
\title{\bf Special subgroups of gyrogroups:\\ Commutators, nuclei and radical$^\ast$\footnote{$^\ast$This is the final version of the manuscript published
in Mathematics Interdisciplinary Research {\bf 1} (2016), pp.
53--68. The published version of the article is accessible via
\href{http://mir.kashanu.ac.ir/article_13907_2153.html}{{http://mir.kashanu.ac.ir/}}.}}
\author{Teerapong Suksumran$^\dag$\footnote{$^\dag$The author was financially supported
by Development and Promotion of Science and Technology
\mbox{Talents} Project (DPST) funded by Institute for Promotion of
Teaching Science and Technology (IPST),
Thailand.}\footnote{\copyright\,2016 Author. This manuscript version
is made available under the CC-BY-NC-ND 4.0 license
(\href{http://creativecommons.org/licenses/by-nc-nd/4.0/}{{http://creativecommons.org/licenses/by-nc-nd/4.0/}}).}\\
Department of Mathematics\\
North Dakota State University\\
Fargo, ND 58105, USA\\
and\\
Department of Mathematics and Computer Science\\
Faculty of Science, Chulalongkorn University\\Bangkok 10330,
Thailand\\[5pt]
\texttt{teerapong.suksumran@gmail.com}}
\date{}
\maketitle

\begin{abstract}
A gyrogroup is a nonassociative group-like structure modelled on the
space of relativistically admissible velocities with a binary
operation given by Einstein's velocity addition law. In this
article, we present a few of groups sitting inside a gyrogroup $G$,
including the commutator subgyrogroup, the left nucleus, and the
radical of $G$. The normal closure of the commutator subgyrogroup,
the left nucleus, and the radical of $G$ are in particular normal
subgroups of $G$. We then give a criterion to determine when a
subgyrogroup $H$ of a finite gyrogroup $G$, where the index
$[G\colon H]$ is the smallest prime dividing $|G|$, is normal in
$G$.
\end{abstract}

\thispagestyle{empty}

\newpage

\section{Introduction}\label{Sec: Intro}
A gyrogroup, discovered by Abraham A. Ungar \cite{AU1997TPI}, is a
nonassociative group-like structure modelled on the space of
relativistically admissible velocities, together with Einstein's
velocity addition \cite{AU2007EVA}. It is remarkable that the
gyrogroup structure appears in various fields such as mathematical
physics \cite{JPSK2015HPM, AU2002THG}, non-Euclidean geometry
\cite{AU2008AHG, AU2010BCE}, group theory \cite{TFAU2000IDG,
TFAU2001GDG}, loop theory \cite{WKHU1998KLG, TSKW2015EGB}, harmonic
analysis \cite{MF2015HAM, MF2014HAE}, abstract algebra
\cite{TSKW2015ITG, TSKW2014LTG}, and analysis \cite{MF2009SCW,
TAOH2015GGS}.

\par This article explores an algebraic aspect of gyrogroups. Recall
that in abstract algebra the following theme recurs: given an object
$X$ and a subobject  $Y$, determine whether the quotient object
$X/Y$ has the same algebraic structure as $X$. It is known, for
instance, that a subgroup $\Xi$ of a group $\Gam$ gives rise to the
quotient group $\Gam/\Xi$ if and only if $\Xi$ is normal in $\Gam$.
Sometimes, it is possible to use information on a normal subgroup
$\Xi$ and on the quotient $\Gam/\Xi$ to obtain information on
$\Gam$. Therefore, determining the normal subgroups of $\Gam$ is
useful for studying properties of $\Gam$ itself. The situation in
gyrogroup theory is analogous. For example, the Lagrange theorem for
finite gyrogroups follows from the fact that every gyrogroup $G$ has
a normal subgroup $\Xi$ such that $G/\Xi$ is a gyrocommutative
gyrogroup \cite[Theorem 4.11]{TFAU2000IDG}. For more details, see
Section 5 of \cite{TSKW2014LTG}. From this point of view, we examine
some normal subgyrogroups of a gyrogroup that form groups under the
gyrogroup operation.

\par For basic knowledge of gyrogroup theory, the reader is referred
to \cite{AU2008AHG, TSKW2015ITG, TSKW2014LTG}. Here is the formal
definition of a gyrogroup.

\begin{definition}[Gyrogroup]\label{def: gyrogroup}
A groupoid $(G,\oplus)$ is a \textit{gyrogroup} if its binary
operation satisfies the following axioms.
\begin{enumerate}
    \item[(G1)] There is an element $0\in G$ such that $0\oplus a =
    a$ for all $a\in G$.
    \item[(G2)] For each $a\in G$, there is an element $b\in G$ such that
$b\oplus a = 0$.
    \item[(G3)] For all $a$, $b\in G$, there is an automorphism
$\gyr{a,b}{}\in\aut{G,\oplus}$ such that
    \begin{equation}\tag{left gyroassociative law} a\oplus (b\oplus c) = (a\oplus b)\oplus\gyr{a,
    b}{c}\end{equation}
    for all $c\in G$.
    \item[(G4)] For all $a$, $b\in G$, $\gyr{a\oplus b, b}{} = \gyr{a,
    b}{}$.\hfill(left loop property)
\end{enumerate}
\end{definition}

\section{The commutator subgyrogroup}
\par Throughout this section, $G$ is an arbitrary gyrogroup unless otherwise stated.

\subsection{The direct product and normal closure}
\par Recall that the intersection of normal subgroups of a group
$\Gam$ is again a normal subgroup of $\Gam$. This result continues
to hold for gyrogroups, as we will see shortly. Because of the
missing of associativity in gyrogroups, it is not straightforward to
\mbox{determine} whether a given subgyrogroup $H$ of a gyrogroup $G$
is normal in $G$. However, according to Theorem \ref{thm: Existence
of normal closure}, the smallest (by inclusion) normal subgyrogroup
of $G$ that contains $H$, called the normal closure of $H$, always
exists. The normal closure of $H$ and $H$ have some common features,
and sometimes it is possible to obtain information about $H$ from
the normal closure of $H$. See for instance Corollary \ref{cor:
commutator subgyrogroup is a subgroup}.

\par Given an indexed family of gyrogroups $\cset{G_i}{i\in I}$, the
\textit{direct product of ${G_i}$}, $i\in I$, denoted by
$\lprod{i\in I}{} G_i$, consists of all functions $f\colon I
\to\lcup{i\in I}{}G_i$ with the property that $f(i)\in G_i$ for all
$i\in I$. For $f, g\in\lprod{i\in I}{}G_i$, define a function
$f\oplus g$ by the equation
\begin{equation}\label{eqn: operation on direct product}
(f\oplus g)(i) = f(i)\oplus g(i),\qquad i\in I.
\end{equation}

\begin{theorem}\label{thm: direct product gyrogroup}
Let $\cset{G_i}{i\in I}$ be an indexed family of gyrogroups. The
direct product $\lprod{i\in I}{}G_i$ with operation defined by $(f,
g)\mapsto f\oplus g$ is a gyrogroup.
\end{theorem}
\begin{proof}
Set $G = \lprod{i\in I}{}G_i$. The zero function, $i\mapsto 0$,
$i\in I$, is a left identity of $G$. For each $f\in G$, the function
$i \mapsto\ominus f(i)$, $i\in I$, is a left inverse of $f$. The
gyroautomorphisms of $G$ are given by
$$(\gyr{f, g}{h})(i) = \gyr{f(i), g(i)}{h(i)},\qquad i\in I,$$
for all $f, g, h\in G$. It is straightforward to check that the
axioms of a gyrogroup are satisfied.
\end{proof}

\begin{theorem}\label{thm: intersection of normal is normal}
Let $\cset{N_i}{i\in I}$ be an indexed family of normal
subgyrogroups of $G$. Then the intersection $\lcap{i\in I}{}N_i$ is
a normal subgyrogroup of $G$.
\end{theorem}
\begin{proof}
For each $i\in I$, there exists a gyrogroup homomorphism $\vphi_i$
of $G$ to a gyrogroup $G_i$ such that $\ker{\vphi_i} = N_i$. Set $H
= \lprod{i\in I}{}G_i$. For each $a\in G$, define a function
$\vphi(a)$ by $\vphi(a)(i) = \vphi_i(a)$ for all $i\in I$. Then
$a\mapsto \vphi(a)$, $a\in G$, defines a gyrogroup homo-morphism
from $G$ to $H$. Direct computation shows that $\ker{\vphi} =
\lcap{i\in I}{}\ker{\vphi_i}$. Hence, $\lcap{i\in I}{}N_i =
\lcap{i\in I}{}\ker{\vphi_i} = \ker{\vphi}\unlhd G$.
\end{proof}

\begin{theorem}\label{thm: Existence of normal closure}
Let $A$ be a nonempty subset of $G$. Then there exists a unique
normal subgyrogroup of $G$, denoted by $\NC{A}$, such that
\begin{enumerate}
    \item\label{item: normal closure inclusion} $A\subseteq\NC{A}$, and
    \item\label{item: normal closure smallest} if $N\unlhd G$ and $A\subseteq N$, then $\NC{A}\subseteq
    N$.
\end{enumerate}
\end{theorem}
\begin{proof}
Set $\cols{A} = \cset{K\subseteq G}{K\unlhd G\textrm{ and
}A\subseteq K}$. By Theorem \ref{thm: intersection of normal is
normal}, $\NC{A} := \lcap{K\in\cols{A}}{}K$ forms a normal
subgyrogroup of $G$ satisfying the two conditions. The uniqueness of
$\NC{A}$ follows from condition \eqref{item: normal closure
smallest}.
\end{proof}

\begin{definition}[Normal closure]
Let $A$ be a nonempty subset of a gyrogroup $G$. The normal
subgyrogroup $\NC{A}$ in Theorem \ref{thm: Existence of normal
closure} is called the {\it normal closure of $A$} or {\it normal
subgyrogroup of $G$ generated by $A$}.
\end{definition}

\par According to Theorem \ref{thm: Existence of normal closure},
the normal closure of $A$ is the smallest (by inclusion) normal
subgyrogroup of $G$ that contains $A$. Note that if $A$ itself is a
normal sub-gyrogroup of $G$, then $\NC{A} = A$. In other words, any
normal subgyrogroup of $G$ equals its normal closure. The concept of
normal closures is needed in studying the commutator subgyrogroup of
a gyrogroup in the next section.

\subsection{Commutators}
\par In this section, we extend the notion of commutators, which is
defined for groups, to gyrogroups. Recall that if $\Gam$ is a group,
then the commutator subgroup of $\Gam$, denoted by $\Gam'$, is the
smallest normal subgroup of $\Gam$ such that the quotient
$\Gam/\Gam'$ is an abelian group. Unlike the situation in group
theory, it is still an open problem whether the commutator
subgyrogroup of a gyrogroup $G$, denoted by $G'$, is normal in $G$.
However, it is true that if $G'$ is normal in $G$, then the quotient
$G/G'$ forms a gyrocommutative gyrogroup. Therefore, we focus
attention on the normal closure of $G'$ instead of $G'$. It turns
out that the normal closure of $G'$ is the smallest normal
subgyrogroup of $G$ such that the quotient $G/\NC{G'}$ is
gyrocommutative. Further, the normal closure of $G'$ (and hence
$G'$) forms a subgroup of $G$, as we will see shortly.

\par Let $G$ be a gyrogroup. Given $a, b\in G$,
define the \textit{commutator of $a$ and $b$}, denoted by $[a, b]$,
by the equation
\begin{equation}\label{eqn: commutator of a and b}
[a, b] = \ominus(a\oplus b)\oplus\gyr{a, b}{(b\oplus a)}.
\end{equation}
Define
\begin{equation}
G' = \gen{[a, b]\colon a, b\in G},
\end{equation}
the subgyrogroup of $G$ generated by commutators of elements from
$G$, called the \textit{\mbox{commutator} subgyrogroup} of $G$. Note
that if $G$ is a gyrogroup with trivial gyroauto-morphisms, then $G$
becomes a group, $[a, b]$ becomes the group-theoretic commutator of
$a$ and $b$, and $G'$ becomes the familiar commutator subgroup of
$G$.

\begin{theorem}\label{thm: property of commutator}
Let $G$ be a gyrogroup. Then the following hold.
\begin{enumerate}
    \item\label{item: identity of [a, b]} For all $a, b\in G$, $[a, b] = 0$ if and only if $a\oplus b = \gyr{a, b}{(b\oplus
    a)}$.
    \item\label{item: identity of -(a+b)} For all $a, b\in G$, $\ominus(a\oplus b) = (\ominus a\ominus b)\oplus[\ominus a, \ominus
    b]$.
    \item\label{item: invariant of [a, b] under homomorphism} If $\vphi$ is a gyrogroup homomorphism of $G$, then $\vphi([a,
    b]) = [\vphi(a), \vphi(b)]$ for all $a, b\in G$.
    \item\label{item: invariant of G' under tau} If $\tau\in\aut{G}$, then $\tau(G') = G'$.
    \item\label{item: characterization of gyrocommutative via G'} $G' = \set{0}$ if and only if $G$ is gyrocommutative.
    \item\label{item: G/G' is gyrocommutative} If $G'\unlhd G$, then $G/G'$ is gyrocommutative.
\end{enumerate}
\end{theorem}
\begin{proof}
Item \eqref{item: identity of [a, b]} follows from the left
cancellation law.

\par To verify item \eqref{item: identity of -(a+b)}, we compute
\begin{eqnarray*}
(\ominus a\ominus b)\oplus[\ominus a, \ominus b] &=& \gyr{\ominus a,
\ominus b}{(\ominus b\ominus a)}\\
{} &=& \gyr{a, b}{(\ominus b\ominus a)}\\
{} &=& \ominus(a\oplus b).
\end{eqnarray*}
We have the first equation from the definition of a commutator; the
second \mbox{equation} from Theorem 2.34 of \cite{AU2008AHG}; and
the last equation from Theorem 2.11 of \cite{AU2008AHG}.

\par \eqref{item: invariant of [a, b] under
homomorphism} By Proposition 23 of \cite{TSKW2015ITG},
\begin{eqnarray*}
\vphi([a, b]) &=& \vphi(\ominus(a\oplus b)\oplus\gyr{a, b}{(b\oplus
a)})\\
{} &=& \ominus(\vphi(a)\oplus \vphi(b))\oplus\gyr{\vphi(a),
\vphi(b)}{(\vphi(b)\oplus \vphi(a))}\\
{} &=& [\vphi(a), \vphi(b)].
\end{eqnarray*}

\par \eqref{item: invariant of G' under tau} Let $\tau\in\aut{G}$.
First, we prove that $G'\subseteq \tau(G')$. For all $a, b\in G$, we
have $[a, b] = \tau([\tau^{-1}(a), \tau^{-1}(b)])$ belongs to
$\tau(G')$. Hence, $\tau(G')$ contains all the commutators of $G$.
Since $G'$ is the smallest subgyrogroup of $G$ containing the
commutators of $G$ and $\tau(G')\leqslant G$, it follows that
$G'\subseteq \tau(G')$. Since $\tau^{-1}$ is also in $\aut{G}$,
$G'\subseteq \tau^{-1}(G')$. This implies $\tau(G')\subseteq
\tau(\tau^{-1}(G')) = G'$ since $\tau$ is a bijection. Hence,
$\tau(G') = G'$.

\par Item \eqref{item: characterization of gyrocommutative via
G'} follows immediately from item \eqref{item: identity of [a, b]}.

\par \eqref{item: G/G' is gyrocommutative} Suppose that $G'\unlhd G$.
Then $G/G'$ has the quotient gyrogroup structure. Let $a, b\in G$.
According to Theorem 27 of \cite{TSKW2015ITG}, we have
\begin{align}
\ominus ((a\oplus G')\oplus(b\oplus G')) &= \ominus((a\oplus
b)\oplus G')\notag\\
{} &= (\ominus(a\oplus b))\oplus G'\notag\\
{} &= ((\ominus a\ominus b)\oplus [\ominus a, \ominus b])\oplus
G'\notag\\
{} &= ((\ominus a\ominus b)\oplus G')\oplus ([\ominus a, \ominus
b]\oplus G')\notag\\
{} &= (\ominus a\ominus b)\oplus G'\notag\\
{} &= (\ominus a\oplus G')\oplus (\ominus b\oplus G')\notag\\
{} &= \ominus (a\oplus G')\ominus (b\oplus G').\notag
\end{align}
This proves that $G/G'$ satisfies the automorphic inverse property
and so $G/G'$ is gyrocommutative by Theorem 3.2 of \cite{AU2008AHG}.
\end{proof}

\par A subgyrogroup $H$ of $G$ is called an {\it L-subgyrogroup} of $G$, denoted
by \mbox{$H\leqslant_L G$}, if $\gyr{a, h}{(H)} = H$ for all $a\in
G$ and $h\in H$. For more information about L-subgyrogroups, see
Section 4 of \cite{TSKW2015ITG}.

\begin{theorem}\label{thm: G' is an L subgyrogroup}
The commutator subgyrogroup of $G$ is an L-subgyrogroup of $G$.
\end{theorem}
\begin{proof}
By Theorem \ref{thm: property of commutator} \eqref{item: invariant
of G' under tau}, $G'$ is invariant under the gyroautomorphisms of
$G$. Hence, $G'\leqslant_L G$.
\end{proof}

\newpage

\begin{proposition}\label{prop: G' is smallest with gyrocommutative quotient}
Let $N$ be a normal subgyrogroup of $G$. The following are
\mbox{equivalent}:
\begin{enumerate}
    \item\label{item: G/N is gyrocomm} $G/N$ is gyrocommutative.
    \item\label{item: N contains G'} $G'\subseteq N$.
    \item\label{item: N contains the commutators} $[a, b]\in N$ for all $a, b\in G$.
\end{enumerate}
\end{proposition}
\begin{proof}
\eqref{item: G/N is gyrocomm} $\Rightarrow$ \eqref{item: N contains
G'} Let $a, b\in G$. Set $X = a\oplus N$ and $Y = b\oplus N$. Since
$G/N$ is gyro-commutative, $X\oplus Y = \gyr{X, Y}{(Y\oplus X)}$.
From Theorem 27 of \cite{TSKW2015ITG}, we have $$(a\oplus b)\oplus N
= (\gyr{a, b}{(b\oplus a)})\oplus N.$$ It follows that
\begin{align}
[a, b]\oplus N &= (\ominus (a\oplus b)\oplus\gyr{a, b}{(b\oplus
a)})\oplus N\notag\\
{} & = \ominus((a\oplus b)\oplus N)\oplus (\gyr{a, b}{(b\oplus
a)}\oplus N)\notag\\
{} &= \ominus((a\oplus b)\oplus N)\oplus ((a\oplus b)\oplus N)\notag\\
{} &= 0\oplus N.\notag
\end{align}
Hence, $[a, b]\in N$ for all $a, b\in G$ and so $G'\subseteq N$ by
the minimality of $G'$.

\par The implication \eqref{item: N contains G'} $\Rightarrow$ \eqref{item: N contains the
commutators} is trivial.

\par \eqref{item: N contains the commutators} $\Rightarrow$ \eqref{item: G/N is
gyrocomm} Since $N\unlhd G$, $G/N$ admits the quotient gyrogroup
structure. The proof that $G/N$ is gyrocommutative follows the same
steps as in the proof of Theorem \ref{thm: property of commutator}
\eqref{item: G/G' is gyrocommutative}.
\end{proof}

\begin{theorem}\label{thm: unique normal subgyrogroup}
The normal closure of $G'$ is the unique normal subgyrogroup of $G$
such that
\begin{enumerate}
    \item\label{item: g/NC(G') is gyrocomm} $G/\NC{G'}$ is gyrocommutative, and
    \item\label{item: universal property of NC(G')} if $\vphi\colon G\to A$ is a gyrogroup homomorphism into
    a gyrocommutative gyrogroup $A$, then $\vphi$ factors through
    $\NC{G'}$ in the sense that $\NC{G'}\subseteq \ker{\vphi}$.
\end{enumerate}
\end{theorem}
\begin{proof}
\par By Theorem \ref{thm: Existence of normal closure},
$\NC{G'}\unlhd G$ and $G'\subseteq \NC{G'}$. Hence, by Proposition
\ref{prop: G' is smallest with gyrocommutative quotient},
$G/\NC{G'}$ is gyrocommutative. Suppose that $\vphi\colon G\to A$ is
a gyrogroup homomorphism of $G$, where $A$ is a gyrocommutative
gyrogroup. For $a, b\in G$, we have $$\vphi([a, b]) = [\vphi(a),
\vphi(b)] = 0$$ since $\vphi(a), \vphi(b)\in A$ and $A$ is
gyrocommutative. Thus, $[a, b]\in\ker{\vphi}$ for all $a, b\in G$,
which implies $G'\subseteq \ker{\vphi}$. Since $\ker{\vphi}\unlhd
G$,  it follows from the minimality of $\NC{G'}$ that
$\NC{G'}\subseteq\ker{\vphi}$.

\par (Uniqueness) Assume that $K_1$ and $K_2$ are normal
subgyrogroups of $G$ that satisfy the two conditions. Let
$\Pi_1\colon G\to G/K_1$ and $\Pi_2\colon G\to G/K_2$ be the
canonical projections. As $K_1$ satisfies the second condition and
$\Pi_2$ is a gyrogroup homomorphism, we have $K_1\subseteq
\ker{\Pi_2} = K_2$. Interchanging the roles of $K_1$ and $K_2$, one
obtains that $K_2\subseteq\ker{\Pi_1} = K_1$. Hence, $K_1 = K_2$.
\end{proof}

\par Theorem \ref{thm: unique normal subgyrogroup} implies the
{\it universal property} of the normal closure of $G'$: given any
gyrogroup homomorphism $\vphi$ from $G$ to a gyrocommutative
gyrogroup $A$, there is a unique gyrogroup homomorphism $\Phi\colon
G/\NC{G'}\to A$ such that $\Phi\circ\Pi = \vphi$, that is, the
following diagram commutes.
$$
\begin{tikzpicture}[description/.style={fill=white,inner sep=2pt}]
\matrix (m) [matrix of math nodes, row sep=0.5em, column sep=2em,
text height=1.5ex, text depth=0.25ex]
{ G & & A \\
& \scalebox{2}{$\circlearrowleft$} & \\
& G/\NC{G'} & \\ }; \path[->] (m-1-1) edge node[auto]{$\vphi$}
(m-1-3) edge node[below left] {$\Pi$} (m-3-2); \path[dotted,<-]
(m-1-3) edge node[auto] {$\Phi$} (m-3-2);
\end{tikzpicture}
$$
\noindent Here, $\Pi$ denotes the canonical projection given by
$\Pi(a) = a\oplus\NC{G'}$ for all $a\in G$, and $\Phi$ is given by
\begin{equation}
\Phi(a\oplus\NC{G'}) = \vphi(a)
\end{equation}
for all $a\in G$.

\begin{theorem}\label{thm: minimality of normal closure of G'}
Let $N$ be a normal subgyrogroup of $G$. Then $G/N$ is
gyrocommutative if and only if $\NC{G'}\subseteq N$.
\end{theorem}
\begin{proof}
Suppose that $G/N$ is gyrocommutative. Then the canonical projection
$\Pi\colon G\to G/N$ fits item \eqref{item: universal property of
NC(G')} of Theorem \ref{thm: unique normal subgyrogroup}. Hence,
$\NC{G'}\subseteq\ker{\Pi} = N$. Conversely, if $\NC{G'}\subseteq
N$, then $G'\subseteq N$ and so $G/N$ is gyrocommutative by
Proposition \ref{prop: G' is smallest with gyrocommutative
quotient}.
\end{proof}

\begin{proposition}\label{prop: K trivial when G is gyrocomm}
$\NC{G'} = \set{0}$ if and only if $G$ is gyrocommutative.
\end{proposition}
\begin{proof}
If $\NC{G'} = \set{0}$, then $G\cong G/\NC{G'}$ via the canonical
projection. Hence, $G$ is gyrocommutative. Conversely, if $G$ is
gyrocommutative, then so is $G/\set{0}$. Hence,
$\NC{G'}\subseteq\set{0}$ by Theorem \ref{thm: minimality of normal
closure of G'}. This implies $\NC{G'}=\set{0}$.
\end{proof}

\par By a {\it subgroup} of a gyrogroup $G$ we mean a subgyrogroup
of $G$ that forms a group under the operation of $G$
\cite[Proposition 3.3]{TSKW2014LTG}. One of the remarkable
consequences of Theorem \ref{thm: unique normal subgyrogroup} is
that the normal closure of $G'$ (and hence $G'$) is a subgroup of
$G$.

\begin{theorem}
The normal closure of $G'$ is a subgroup of $G$.
\end{theorem}
\begin{proof}
By Theorem 4.11 of \cite{TFAU2000IDG}, $G$ has a normal subgroup
$\Xi$ such that $G/\Xi$ is a gyrocommutative gyrogroup. By Theorem
\ref{thm: minimality of normal closure of G'},
$\NC{G'}\subseteq\Xi$. Since $\Xi$ is a subgroup of $G$, so is
$\NC{G'}$.
\end{proof}

\begin{corollary}\label{cor: commutator subgyrogroup is a subgroup}
The commutator subgyrogroup of $G$ is a subgroup of $G$.
\end{corollary}
\begin{proof}
The corollary follows from the fact that $G'\subseteq\NC{G'}$.
\end{proof}

\section{Nuclei and the radical of a gyrogroup}
\par Throughout this section, $G$ is an arbitrary gyrogroup.
We follow \cite{TFMKJP2006OTS} in presenting a few normal subgroups
sitting inside a gyrogroup. The main goal of this section is to
prove that the left nucleus and radical of $G$ are normal
subgyrogroups of $G$ that form groups under the gyrogroup operation.
The key idea is as follows. Every gyrogroup can be embedded into its
left multiplication group, and normality of the subgyrogroup under
consideration follows from normality of the corresponding subgroup
of the left multiplication group. This in particular shows a
remarkable connection between groups and gyrogroups.

\par As in loop theory, the \textit{left nucleus}, \textit{middle nucleus}, and \textit{right
nucleus of $G$} are defined, respectively, by
\begin{align}
N_l(G) &= \cset{a\in G}{\forall b, c\in G,\,\,a\oplus(b\oplus c) = (a\oplus b)\oplus c},\notag\\
N_m(G) &= \cset{b\in G}{\forall a, c\in G,\,\,a\oplus(b\oplus c) = (a\oplus b)\oplus c},\notag\\
N_r(G) &= \cset{c\in G}{\forall a, b\in G,\,\,a\oplus(b\oplus c) =
(a\oplus b)\oplus c}.\notag
\end{align}
Since $G$ satisfies the left gyroassociative law and the general
left cancellation law, the left nucleus, middle nucleus, and right
nucleus of $G$ can be restated in terms of gyro-automorphisms as
follows:
\begin{align}
N_l(G) &= \cset{a\in G}{\forall b\in G,\,\,\gyr{a, b}{} = \id{G}},\notag\\
N_m(G) &= \cset{b\in G}{\forall a\in G,\,\,\gyr{a, b}{} = \id{G}},\notag\\
N_r(G) &= \cset{c\in G}{\forall a, b\in G,\,\,\gyr{a, b}{c} =
c}.\notag
\end{align}
By Theorem 2.34 of \cite{AU2008AHG}, $\igyr{a, b}{} = \gyr{b, a}{}$
for all $a, b\in G$. It follows that the left nucleus and middle
nucleus of $G$ are identical.

\begin{theorem}\label{thm: l m r nucleus form subgroups}
The left nucleus, middle nucleus, and right nucleus are
L-subgyrogroups of $G$. Furthermore, they are subgroups of $G$.
\end{theorem}
\begin{proof}
Because $\gyr{0, a}{} = \id{G}$ for all $a\in G$, $0\in N_l(G)$. Let
$a\in N_l(G)$ and let $b\in G$. By Theorem 2.34 of \cite{AU2008AHG},
$\gyr{\ominus a, b}{} = \gyr{\ominus a, \ominus(\ominus b)}{} =
\gyr{a, \ominus b}{} = \id{G}$. Hence, $\ominus a$ is in $N_l(G)$.
Let $a, b\in N_l(G)$ and let $c, x\in G$. According to the gyrator
identity \cite[Theorem 2.10]{AU2008AHG}, we compute
\begin{align}
\gyr{a\oplus b, c}{x} &= \ominus((a\oplus b)\oplus c)\oplus((a\oplus
b)\oplus(c\oplus x))\notag\\
{} &= \ominus((a\oplus b)\oplus c)\oplus(a\oplus (b\oplus \gyr{b,
a}{(c\oplus x)}))\notag\\
{} &= \ominus((a\oplus b)\oplus c)\oplus(a\oplus(b\oplus(c\oplus
x)))\notag\\
{} &= \ominus((a\oplus b)\oplus c)\oplus(a\oplus ((b\oplus c)\oplus
x))\notag\\
{} &= \ominus((a\oplus b)\oplus c)\oplus((a\oplus (b\oplus c))\oplus
x)\notag\\
{} &= \ominus((a\oplus b)\oplus c)\oplus(((a\oplus b)\oplus c)\oplus
x)\notag\\
{} &= x.\notag
\end{align}
We have the second equation from the right gyroassociative law; the
third and forth equations since $b\in N_l(G)$; the fifth and sixth
equations since $a\in N_l(G)$; the last \mbox{equation} from the
left cancellation law. Since $x$ is arbitrary, $\gyr{a\oplus b, c}{}
= \id{G}$ and so $a\oplus b\in N_l(G)$. By the subgyrogroup
criterion \cite[Proposition 14]{TSKW2015ITG}, $N_l(G)\leqslant G$.
By definition of $N_l(G)$, $N_l(G)\leqslant_L G$. Since
$\res{\gyr{a, b}{}}{N_l(G)} = \id{N_l(G)}$ for all $a, b\in N_l(G)$,
$N_l(G)$ is a subgroup of $G$. Since $N_m(G) = N_l(G)$, we have
$N_m(G)\leqslant_L G$ and $N_m(G)$ is a subgroup of $G$ as well. The
proof that $N_r(G)$ is an L-subgyrogroup and a subgroup of $G$ is
straightforward.
\end{proof}

\par Let $a$ be an arbitrary element of $G$. Recall that
the \textit{left gyrotranslation by $a$}, $L_a$, is a permutation of
$G$ defined by
$$L_a(x) = a\oplus x,\qquad x\in G.$$
For a given subgyrogroup $H$ of $G$, define $L(H)=\cset{L_a}{a\in
H}$. In the case $H = G$, we have $L(G) = \cset{L_a}{a\in G}$. The
\textit{left multiplication group of $G$}, $\lmlt{G}$, is the
subgroup of the symmetric group on $G$ generated by $L(G)$. In other
words,
$$\lmlt{G} = \gen{L_a\colon a\in G}.$$

\par A subset $X$ of a group $\Gam$ is a \textit{twisted subgroup}
\cite[p. 187]{TFMKJP2006OTS} of $\Gam$ if $1\in X$, $1$ being the
identity element of $\Gam$; $x\in X$ implies $x^{-1}\in X$; and $x,
y\in X$ implies $xyx\in X$.

\begin{theorem}\label{thm: L(G) is a twisted subgroup}
$L(G)$ is a twisted subgroup of $\lmlt{G}$.
\end{theorem}
\begin{proof}
The theorem follows from the fact that $L_a^{-1} = L_{\ominus a}$
and $$L_a\circ L_b\circ L_a = L_{(a\oplus b)\boxplus a}$$ for all
$a, b\in G$. Here, the coaddition $\boxplus$ of $G$ is defined by
$a\boxplus b = a\oplus\gyr{a, \ominus b}{b}$ for all $a, b\in G$.
\end{proof}

\par In light of Theorem \ref{thm: L(G) is a twisted subgroup},
$L(G)$ is a {\it generating} twisted subgroup of $\lmlt{G}$. This
leads to the following theorem.

\begin{theorem}\label{thm: Description of L(G) sharp}
Define
$$
L(G)^\# = \lcap{a\in G}{}L_aL(G).
$$
Then $L(G)^\#$ is a normal subgroup of $\lmlt{G}$ contained in
$L(G)$.
\end{theorem}
\begin{proof}
The theorem is an application of Theorem 3.8 of
\cite{TFMKJP2006OTS}.
\end{proof}

\begin{theorem}\label{thm: L(G) sharp and L(left nucleus)}
$L(N_l(G))$ is a normal subgroup of $\lmlt{G}$.
\end{theorem}
\begin{proof}
From Theorem 5.7 of \cite{TFMKJP2006OTS}, we have $L(N_l(G)) =
L(G)^\#$. Hence, $L(N_l(G))$ is a normal subgroup of $\lmlt{G}$ by
Theorem \ref{thm: Description of L(G) sharp}.
\end{proof}

\par Following \cite{TFMKJP2006OTS}, we define
\begin{equation}\label{eqn: L(G)'}
L(G)' = \cset{L_{a_1}\circ L_{a_2}\circ\cdots\circ L_{a_n}}{a_i\in
G\textrm{ and }L_{a_n}\circ L_{a_{n-1}}\circ\cdots\circ L_{a_1} =
\id{G}}.
\end{equation}
Since $L(G)$ is a generating twisted subgroup of $\lmlt{G}$, it
follows from a result of Foguel, Kinyon, and Phillips \cite[p.
189]{TFMKJP2006OTS} that $L(G)'$ is a normal subgroup of $\lmlt{G}$.
In fact, we have the following theorem.

\begin{theorem}\label{thm: properties of L(G)'}
$L(G)'$ is a normal subgroup of $\lmlt{G}$ such that $L(G)'\subseteq
L(G)^\#$.
\end{theorem}
\begin{proof}
The theorem follows directly from Proposition 3.10 of
\cite{TFMKJP2006OTS}.
\end{proof}

\par Set $\symz{G} = \cset{\sig\in\sym{G}}{\sig(0) = 0}$.
Note that $L(G)\cap\symz{G} = \set{\id{G}}$. This implies that if
$G$ is a gyrogroup with a nonidentity gyroautomorphism, say $\gyr{a,
b}{}$, then $L(G)$ is a proper twisted subgroup of $\lmlt{G}$. In
fact, $\gyr{a, b}{} = L^{-1}_{a\oplus b}\circ L_a\circ L_b$ belongs
to $\lmlt{G}$, but does not belong to $L(G)$ since otherwise
$\gyr{a, b}{}$ would \mbox{belong} to $L(G)\cap\symz{G} =
\set{\id{G}}$. In this case, $L(G)'$ and $L(G)^\#$ form {\it proper}
normal subgroups of $\lmlt{G}$ for they are contained in $L(G)$.

\par The following proposition provides a sufficient
condition for normality of a subgyrogroup. As an application of this
proposition, we prove that the left nucleus and radical of $G$ are
normal subgyrogroups of $G$.

\begin{proposition}[{\hskip-0.3pt}\cite{TS2015TAG}]\label{prop: sufficient condition to be normal, subgyrogroup}
If $H$ is a subgyrogroup of $G$ such that
\begin{enumerate}
\item\label{item: sufficient condition, L} $\gyr{h, a}{} = \id{G}$ for all $h\in H, a\in G$,
\item\label{item: sufficient condition, invariant under gyr} $\gyr{a, b}{(H)}\subseteq H$ for all $a, b\in G$, and
\item\label{item: sufficient condition, left and right coset} $a\oplus H = H\oplus a$ for all $a\in G$,
\end{enumerate}
then $H$ is a normal subgyrogroup of $G$.
\end{proposition}

\begin{lemma}\label{lem: properties of left nucleus}
Let $G$ be a gyrogroup. Then
\begin{enumerate}
    \item $\gyr{a, b}{(N_l(G))} \subseteq N_l(G)$ for all $a, b\in
    G$, and
    \item $N_l(G)\oplus a = a\oplus N_l(G)$ for all $a\in G$.
\end{enumerate}
\end{lemma}

\begin{proof}
(1) Set $N = N_l(G)$ and let $n$ be an arbitrary element of $N$. Let
$a, b\in G$. \mbox{According} to the commutation relation
\cite[Equation (14)]{TSKW2015ITG}, we have $$L_{\gyr{a, b}{n}} =
\gyr{a, b}{}\circ L_n\circ\igyr{a, b}{}.$$ Since $\gyr{a,
b}{}\in\lmlt{G}$ and $L(N)\unlhd \lmlt{G}$, it follows that
$L_{\gyr{a, b}{n}}$ belongs to $L(N)$. Hence, $L_{\gyr{a, b}{n}} =
L_{\tilde{n}}$ for some $\tilde{n}\in N$, which implies $\gyr{a,
b}{n} = \tilde{n}\in N$. Since $n$ is arbitrary, we obtain $\gyr{a,
b}{(N)}\subseteq N$.

\par (2) Let $a\in G$ and let $n\in N$. By the left cancellation law,
$x = \ominus a\oplus(n\oplus a)$ is such that $n\oplus a = a\oplus
x$. We compute
\begin{align}
L_x &= L_{\ominus a\oplus (n\oplus a)}\notag\\
{} &= L_{(\ominus a\oplus n)\oplus a}\notag\\
{} &= L_{\ominus a\oplus n}\circ L_a\circ\igyr{\ominus a\oplus n,
a}{}\notag\\
{} &= L_{\ominus a\oplus n}\circ L_a\circ\igyr{\ominus a\oplus n,
a\oplus(\ominus a\oplus n)}{}\notag\\
{} &= L_{\ominus a\oplus n}\circ L_a\circ\igyr{\ominus a\oplus n,
n}{}\notag\\
{} &= L_{\ominus a\oplus n}\circ L_a\notag\\
{} &= L_{\ominus a}\circ L_n\circ \igyr{\ominus a, n}{}\circ
L_a\notag\\
{} &= L_a^{-1}\circ L_n\circ L_a.\notag
\end{align}
We obtain the second equation since $n\in N = N_m(G)$; the third and
seventh equations from the identity $L_{a\oplus b} = L_a\circ L_b
\circ \igyr{a, b}{}$; the forth equation from the right loop
property; the sixth and last equations since $n\in N$. Since
$L(N)\unlhd\lmlt{G}$, we have $L_x\in L(N)$, which implies $x\in N$.
Thus, $N\oplus a\subseteq a\oplus N$.

\par From Lemma 2.19 of \cite{AU2008AHG}, we can let $y\in G$
be such that $a\oplus n = y\oplus a$. To conclude that $a\oplus
N\subseteq N\oplus a$, we have to show that $y$ belongs to $N$. In
fact, one obtains similarly that $L_n = L_a^{-1}\circ L_y\circ L_a$,
which implies $L_y = L_a\circ L_n\circ L_a^{-1}\in L(N)$. Hence,
$y\in N$, as desired.
\end{proof}

\begin{theorem}\label{thm: left nucleus is a normal subgroup}
The left nucleus of $G$ is a normal subgroup of $G$.
\end{theorem}
\begin{proof}
The theorem follows immediately from Theorem \ref{thm: l m r nucleus
form subgroups}, Proposition \ref{prop: sufficient condition to be
normal, subgyrogroup}, the de-fining property of $N_l(G)$, and Lemma
\ref{lem: properties of left nucleus}.
\end{proof}

\begin{corollary}
The middle nucleus of $G$ is a normal subgroup of $G$.
\end{corollary}
\begin{proof}
This is because the left nucleus and middle nucleus of $G$ are the
same.
\end{proof}

\par Following \cite{TFMKJP2006OTS}, the \textit{radical of $G$},
denoted by $\rad{G}$, is defined by
\begin{equation}
\rad{G} = \cset{a\in G}{L_a\in L(G)'}.
\end{equation}

\begin{theorem}\label{thm: radical is a subgroup}
The radical of $G$ is a subgroup of $G$ contained in the left
nucleus of $G$.
\end{theorem}
\begin{proof}
First, we prove that $\rad{G}\subseteq N_l(G)$. Let $a\in\rad{G}$.
Then $L_a\in L(G)'$. By Theorem \ref{thm: properties of L(G)'},
$L(G)'\subseteq L(G)^{\#}$ and by Theorem 5.7 of
\cite{TFMKJP2006OTS}, $L(G)^{\#} = L(N_l(G))$. It follows that
$L_a\in L(N_l(G))$, which implies $a\in N_l(G)$.
\par Let $a\in \rad{G}$. Then $L_{\ominus a} = L_a^{-1}\in L(G)'$
for $L(G)'\leqslant\lmlt{G}$. Hence, $\ominus a\in\rad{G}$. Let $a,
b\in\rad{G}$. Since $\rad{G}\subseteq N_l(G)$, $\gyr{a, b}{} =
\id{G}$. Thus, $$L_{a\oplus b} = L_a\circ L_b\circ\igyr{a, b}{} =
L_a\circ L_b$$ belongs to $L(G)'$. This proves $a\oplus b\in
\rad{G}$ and by the subgyrogroup criterion, $\rad{G}\leqslant G$.
Since $N_l(G)$ is a subgroup of $G$, so is $\rad{G}$.
\end{proof}

\begin{lemma}\label{lem: properties of radical}
Let $G$ be a gyrogroup. Then
\begin{enumerate}
    \item $\gyr{a, b}{(\rad{G})} \subseteq \rad{G}$ for all $a, b\in
    G$, and
    \item $\rad{G}\oplus a = a\oplus \rad{G}$ for all $a\in G$.
\end{enumerate}
\end{lemma}
\begin{proof}
The proof of this lemma follows the same steps as in the proof of
Lemma \ref{lem: properties of left nucleus} with appropriate
modifications.
\end{proof}

\begin{theorem}
The radical of $G$ is a normal subgroup of $G$.
\end{theorem}
\begin{proof}
The theorem follows directly from Proposition \ref{prop: sufficient
condition to be normal, subgyrogroup}, Theorem \ref{thm: radical is
a subgroup}, and Lemma \ref{lem: properties of radical}.
\end{proof}

\par Note that a gyrogroup $G$ is a group if and only if
$G$ equals its left nucleus. Hence, if $G$ is a gyrogroup that is
not a group, then $N_l(G)$ and $\rad{G}$ are {\it proper} normal
subgroups of $G$. Note also that normality of $N_l(G)$ and $\rad{G}$
in $G$ follows from normality of $L(N_l(G))$ and $L(L(G)')$ in the
left multiplication group of $G$, see the proof of Lemma \ref{lem:
properties of left nucleus}.

\section{Subgyrogroups of prime index}
\par Motivated by the study of subgroups of prime index in
\cite{TL2004OSP}, we study subgyrogroups of prime index.
Specifically, we are going to prove a gyrogroup version of the
following well-known result in abstract algebra: if $\Xi$ is a
subgroup of a finite group $\Gam$ such that the index
$[\Gam\colon\Xi]$ is the smallest prime dividing the order of
$\Gam$, then $\Xi$ is normal in $\Gam$ \cite[Theorem 1]{TL2004OSP}.
It is notable that normality of a subgyrogroup $H$ of a finite
gyrogroup $G$, where $[G\colon H]$ is the smallest prime dividing
the order of $G$, depends on the invariance of the left cosets of
$H$ in $G$ under the gyroautomorphisms of $G$, see Theorem \ref{thm:
Characterization of normality invariant under gyromap}.

\par Unless stated otherwise, $G$ is an arbitrary finite gyrogroup.

\par Let $G$ be a gyrogroup, let $a\in G$, and let
$m\in\Z$. Define recursively the following notation:
\begin{equation}
0a = 0,\hskip0.3cm ma = a\oplus ((m-1)a),\, m \geq 1,\hskip0.3cm ma
= (-m)(\ominus a),\, m < 0.
\end{equation}
By induction, one can verify the following usual rules of integral
multiples:
\begin{enumerate}
    \item $(-m)a = \ominus(ma) = m(\ominus a)$,
    \item $(m+k)a = (ma)\oplus(ka)$, and
    \item $(mk)a = m(ka)$
\end{enumerate}
for all $a\in G$ and $m, k\in\Z$.

\begin{theorem}\label{thm: Equivalence of subgyrogroup of prime index}
Suppose that $H$ is a subgyrogroup of a gyrogroup $G$ such that
$[G\colon H] = p$, $p$ being a prime. The following are equivalent:
\begin{enumerate}
    \item\label{item: pa in H} For any $a\in G\setminus H$, $pa\in
    H$.
    \item\label{item: na in H} For any $a\in G\setminus H$, $na\in
    H$ for some positive integer $n$, depending on $a$, with no
    prime divisor less than $p$.
    \item\label{item: a, 2a, 3a, not in H} For any $a\in G\setminus
    H$, $a, 2a,\dots, (p-1)a\not\in H$.
\end{enumerate}
\end{theorem}
\begin{proof}
\eqref{item: pa in H} $\Rightarrow$ \eqref{item: na in H} Choosing
$n = p$ gives item \eqref{item: na in H}.

\par \eqref{item: na in H} $\Rightarrow$ \eqref{item: a, 2a, 3a, not in
H} Let $a\in G\setminus H$ and let $n$ be as in item \eqref{item: na
in H}. By the well-ordering principle, we can let $s$ be the
smallest positive integer such that $sa\in H$. Note that $s > 1$.
Write $n = st+r$ with $0\leq r <s$. Then $ra = (n-st)a =
(na)\oplus(-st)a = (na)\ominus(t(sa))$. Thus, $ra\in H$ for $na,
sa\in H$. The minimality of $s$ forces $r = 0$, so $n = st$. If $s <
p$, then $s$ (and hence $n$) would have a prime divisor less than
$p$. Hence, $s\geq p$, which implies $a, 2a,\dots, (p-1)a\not\in H$.

\par \eqref{item: a, 2a, 3a, not in H} $\Rightarrow$ \eqref{item: pa in
H} First, we prove that $0\oplus H, a\oplus H,\dots, (p-1)a\oplus H$
are all distinct. Assume to the contrary that $ra\oplus H = sa\oplus
H$ for some integers $r$ and $s$ such that \mbox{$0\leq r < s\leq
p-1$}. Then $sa = (ra)\oplus h$ for some $h\in H$. It follows that
$(-r+s)a = \ominus(ra)\oplus(sa) = h\in H$. This contradicts the
assumption because $0 < s-r < p$.

\par Since $\abs{G/H} = [G\colon H] = p$, we have
$G/H = \set{0\oplus H, a\oplus H,\dots, (p-1)a\oplus H}$. Hence,
$pa\oplus H = ta\oplus H$ for some $t$ with $0\leq t\leq p-1$. As
before, the equality gives $(p-t)a\in H$. Since $0\leq t\leq p-1$,
we have $1\leq p-t \leq p$. By assumption, $p-t = p$, which implies
$t = 0$. Hence, $pa\oplus H = 0\oplus H = H$ and so $pa\in H$.
\end{proof}

\begin{proposition}\label{prop: Completed description of G/H}
Let $H$ be a subgyrogroup of a gyrogroup $G$ such that $[G\colon H]
= p$, $p$ being a prime. If $H$ satisfies one of the conditions in
Theorem \ref{thm: Equivalence of subgyrogroup of prime index}, then
$$G/H = \set{0\oplus H, a\oplus H,\dots, (p-1)a\oplus H}$$
for any $a\in G\setminus H$.
\end{proposition}
\begin{proof}
There is no loss in assuming that $H$ satisfies condition
\eqref{item: a, 2a, 3a, not in H} of Theorem \ref{thm: Equivalence
of subgyrogroup of prime index}. As proved in Theorem \ref{thm:
Equivalence of subgyrogroup of prime index}, $G/H = \set{0\oplus H,
\dots, (p-1)a\oplus H}$ for any $a\not\in H$.
\end{proof}

\begin{proposition}\label{prop: Smallest prime implies one of condition}
Let $H$ be a subgyrogroup of $G$. If $[G\colon H]$ is the smallest
prime \mbox{dividing} the order of $G$, then $H$ satisfies condition
\eqref{item: na in H} of Theorem \ref{thm: Equivalence of
subgyrogroup of prime index}.
\end{proposition}
\begin{proof}
By Proposition 6.1 of \cite{TSKW2014LTG}, $\abs{G}a = 0\in H$. Since
$\abs{G}$ has no prime divisors less than $[G\colon H]$, condition
\eqref{item: na in H} of Theorem \ref{thm: Equivalence of
subgyrogroup of prime index} holds.
\end{proof}

\begin{theorem}\label{thm: Characterization of normality invariant under gyromap}
Let $H$ be a subgyrogroup of $G$ such that $[G\colon H]$ is the
smallest prime dividing the order of $G$. Then $H\unlhd G$ if and
only if there is an element $y\in G\setminus H$ such that
$$\gyr{a, b}{(iy\oplus H)}\subseteq iy\oplus H$$
for all $a, b\in G$ and $i\in\set{0,1,\dots,p-1}$.
\end{theorem}
\begin{proof}Set $[G\colon H] = p$.
\par ($\Rightarrow$) Suppose that $H\unlhd G$. Then $G/H$ admits the
gyrogroup structure and becomes a gyrogroup of order $p$. By Theorem
6.2 of \cite{TSKW2014LTG}, $G/H$ forms a cyclic group. In
particular, $\gyr{X, Y}{Z} = Z$ for all $X, Y, Z\in G/H$. Let $a, b,
c$ be arbitrary elements of $G$. Set $X = a\oplus H$ and $Y =
b\oplus H$. From Theorem 27 of \cite{TSKW2015ITG}, we have $c\oplus
H = \gyr{X, Y}{(c\oplus H)} = (\gyr{a, b}{c})\oplus H$. Since
$H\unlhd G$, $\gyr{a, b}{(H)} = H$, which implies $(\gyr{a,
b}{c})\oplus H = \gyr{a, b}{(c\oplus H)}$. Hence, $\gyr{a,
b}{(c\oplus H)} = c\oplus H$.

\par ($\Leftarrow$) Let $y$ be as in the assumption.
By Propositions \ref{prop: Completed description of G/H} and
\ref{prop: Smallest prime implies one of condition}, $$G/H =
\set{0\oplus H, y\oplus H,\dots, (p-1)y\oplus H}.$$ For each $x\in
G$, $x\oplus H = iy\oplus H$ for some $i\in\set{0, 1, \dots, p-1}$.
By assumption,
$$\gyr{a, b}{(x\oplus H)} = \gyr{a, b}{(iy\oplus H )}\subseteq
iy\oplus H = x\oplus H.$$ By Theorem 4.5 of \cite{TS2015GAG}, $G$
acts on $G/H$ by left gyroaddition. By \mbox{Proposition 3.5 (2)}
and Theorem 4.6 of \cite{TS2015GAG}, $\ker{\dot\vphi}\subseteq H$,
where $\dot\vphi$ is the associated permutation
\mbox{representation} of $G$. By the first isomorphism theorem
\cite[Theorem 28]{TSKW2015ITG}, $$G/\ker{\dot\vphi}\cong
\im{\dot\vphi}\leqslant \sym{G/H}.$$ Hence,
$[G\colon\ker{\dot\vphi}]$ divides $p!$. Since
$\ker{\dot\vphi}\leqslant_L G$ and $H\leqslant_L G$, we have
$$[G\colon\ker{\dot\vphi}] = [G\colon H][H\colon \ker{\dot\vphi}] =
p[H\colon \ker{\dot\vphi}],$$ which implies
$[H\colon\ker{\dot\vphi}]$ divides $(p-1)!$. If
$[H\colon\ker{\dot\vphi}] > 1$, one would find a prime $q$ dividing
$[H\colon\ker{\dot\vphi}]$ and would have $q|(p-1)!$. Thus, $q < p$
and $q$ divides $\abs{H}$. Since $\abs{H}$ divides $\abs{G}$, we
have $q$ divides $\abs{G}$, a contradiction. Hence,
$[H\colon\ker{\dot\vphi}] = 1$ and so $H = \ker{\dot\vphi}\unlhd G$.
\end{proof}

\par Recall from abstract algebra that a subgroup of a group $\Gam$
of index two is normal in $\Gam$. This result can be generalized to
the case of gyrogroups as follows.

\begin{theorem}\label{thm: Subgyrogroup of index 2}
If $H$ is a subgyrogroup of $G$ such that $\gyr{a, b}{(H)}\subseteq
H$ for all $a, b\in G$ and $[G\colon H] = 2$, then $H\unlhd G$.
\end{theorem}
\begin{proof}
Let $y\in G\setminus H$ be fixed. By Propositions \ref{prop:
Completed description of G/H} and \ref{prop: Smallest prime implies
one of condition}, \mbox{$G/H = \set{H, y\oplus H}$}. To complete
the proof, we show that $\gyr{a, b}{(y\oplus H)}\subseteq y\oplus H$
for all $a, b\in G$. If $z\in \gyr{a, b}{(y\oplus H)}$, then $z =
\gyr{a, b}{(y\oplus h)} = (\gyr{a, b}{y})\oplus(\gyr{a, b}{h})$ for
some $h\in H$. By assumption, $z\in (\gyr{a, b}{y})\oplus H$. Note
that $\gyr{a, b}{y}\not\in H$ since otherwise $\gyr{a, b}{y} =
\tilde{h}\in H$ would imply $y = \igyr{a, b}{\tilde{h}} = \gyr{b,
a}{\tilde{h}}\in H$, a contradiction. Hence, $(\gyr{a, b}{y})\oplus
H = y\oplus H$ and so $z\in y\oplus H$. This proves $\gyr{a,
b}{(x\oplus H)}\subseteq x\oplus H$ for all $a, b, x\in G$. By
Theorem \ref{thm: Characterization of normality invariant under
gyromap}, $H\unlhd G$.
\end{proof}

\section*{Acknowledgments}
The author is grateful to the Department of Mathematics, North
Dakota State University where he was a visiting fellow during this
work. He would like to thank \mbox{Professor} Abraham Ungar for his
hospitality. This work was completed with the support of Development
and Promotion of Science and Technology Talents Project (DPST),
Institute for Promotion of Teaching Science and Technology (IPST),
\mbox{Thailand}.

\bibliographystyle{amsplain}\addcontentsline{toc}{section}{References}
\bibliography{References}
\end{document}